\documentclass[letterpaper,11pt]{amsart}

\usepackage[all]{xy}                        %
\usepackage{color}
\CompileMatrices                            

\UseTips                                    

\input xypic
\usepackage[bookmarks=true]{hyperref}       

\usepackage{amssymb,latexsym,amsmath,amscd}
\usepackage{xspace}

\usepackage{graphicx}

\reversemarginpar

\theoremstyle{plain}
\newtheorem{theorem}{Theorem}[section]
\newtheorem*{theorem*}{Theorem}
\newtheorem{proposition}[theorem]{Proposition}
\newtheorem{corollary}[theorem]{Corollary}
\newtheorem{lemma}[theorem]{Lemma}

\theoremstyle{definition}
\newtheorem{definition}[theorem]{Definition}

\newtheorem{remark}[theorem]{Remark}
\newtheorem{example}[theorem]{Example}

\newcommand{\enm}[1]{\ensuremath{#1}}          %
\newcommand{\op}[1]{\operatorname{#1}}
\newcommand{\cal}[1]{\mathcal{#1}}

\newcommand{\NN}{\enm{\mathbb{N}}}

\newcommand{\ZZ}{\enm{\mathbb{Z}}}

\newcommand{\HH}{\enm{\mathbb{H}}}

\newcommand{\PP}{\enm{\mathbb{P}}}

\newcommand{\WW}{\enm{\mathbb{W}}}

\newcommand{\Aa}{\enm{\cal{A}}}

\newcommand{\Ee}{\enm{\cal{E}}}
\newcommand{\Ff}{\enm{\cal{F}}}
\newcommand{\Gg}{\enm{\cal{G}}}
\newcommand{\Hh}{\enm{\cal{H}}}
\newcommand{\Ii}{\enm{\cal{I}}}

\newcommand{\Ll}{\enm{\cal{L}}}

\newcommand{\Oo}{\enm{\cal{O}}}

\newcommand{\Rr}{\enm{\cal{R}}}

\renewcommand{\phi}{\varphi}
\renewcommand{\theta}{\vartheta}
\renewcommand{\epsilon}{\varepsilon}

\newcommand{\Ext}{\op{Ext}}

\renewcommand{\to}[1][]{\xrightarrow{\ #1\ }}

\newcommand{\old}[1]{}

\begin{document}

\title[Globally generated vector bundles]{Globally Generated Vector Bundles on the Segre Threefold with Picard Number Two}
\author{E. Ballico, S. Huh and F. Malaspina}
\address{Universit\`a di Trento, 38123 Povo (TN), Italy}
\email{edoardo.ballico@unitn.it}
\address{Sungkyunkwan University, Suwon 440-746, Korea}
\email{sukmoonh@skku.edu}
\address{Politecnico di Torino, Corso Duca degli Abruzzi 24, 10129 Torino, Italy}
\email{francesco.malaspina@polito.it}
\thanks{The first and third authors are partially supported by MIUR and GNSAGA of INDAM (Italy). The second author is supported by Basic Science Research Program 2010-0009195 through NRF funded by MEST. The third author is supported by the framework of PRIN 2010/11 \lq Geometria delle variet\`a algebriche\rq, cofinanced by MIUR}

\begin{abstract}
We classify globally generated vector bundles on $\PP^1 \times \PP^2$ with small first Chern class, i.e. $c_1= (a,b)$, $a+b \leq 3$. Our main method is to investigate the associated smooth curves to globally generated vector bundles via the Hartshorne-Serre construction.
\end{abstract}

\maketitle
\section{Introduction}
Globally generated vector bundles on projective varieties play an important role in classical algebraic geometry. If they are nontrivial they must have strictly positive first Chern class. The classification of globally generated vector bundles with low first Chern class has been done over several rational varieties such as projective spaces \cite{am,SU} and quadric hypersurfaces \cite{BHM}. There is also a recent work over complete intersection Calabi-Yau threefolds by the authors \cite{BHM+++}.

There are three types of Segre varieties of dimension $3$: $\PP^3$, $\PP^1 \times \PP^2$ and $\PP^1 \times \PP^1 \times \PP^1$. In this paper we examine the similar problem of classification of globally generated vector bundles for the Segre variety $\PP^1 \times \PP^2$, the product of a projective line and a projective plane. Note that the classification is already dealt in the case of $\PP^3$ in \cite{am,SU}.

The Hartshorne-Serre correspondence states that the construction of vector bundles of rank $r\ge 2$ on a smooth variety $X$ with dimension $ 3$ is closely related with the structure of curves in $X$ and it inspires the classification of vector bundles on smooth projective threefolds. There have been several works on the classification of {\it arithmetically Cohen-Macaulay} (ACM) bundles on the Segre threefolds \cite{CMP,CFM0,CFM} and so it is sufficiently timely to classify the globally generated vector bundles on the Segre threefolds.

Our main result is to classify globally generated vector bundles on $\PP^1 \times \PP^2$ with low first Chern class $c_1$, up to trivial factor. When $c_1=(1,2)$ or $c_1=(2,1)$, we have the following result (see Theorem \ref{cca5} and Theorem \ref{ttt2} for a more general form in which for some of the invariants $(s;f_1,f_2;r)$ there bundles are described, up to isomorphisms):

\begin{theorem}
Let $\Ee$ be a globally generated vector bundle of rank $r$ at least $2$ on $X=\PP^1 \times \PP^2$ with the first Chern class $c_1=(1,2)$ or $c_1=(2,1)$ and the second Chern class $c_2(\Ee)=(f_1, f_2)$. If $\Ee$ has no trivial factor, then the quadruple $(s;f_1, f_2; r)$ is one of the following:\\
($s$ is the number of connected components of associated curve.)
$$ \left\{
                                           \begin{array}{lll}
                                             (1;1,1;2),~~(1;0,2;2),\\
                                             (1;4,4;3\le r\le 11),~~(1;3,3;3\le r \le 7),~~(1;2,3;3\le r \le 8), \\
                                             (1;c,2;3\le r \le c+2) \text{ with } c\in \{1,2,3,4\}\\
                                            \end{array}
                                         \right.$$
                                         when $c_1=(1,2)$; and
                                 $$ \left\{
                                           \begin{array}{ll}
                                             (s;0,s;r) \text{ with } 1\le s \le 3 \text{  and } 2 \le r \le s+1,\\
                                             (1;1,b;r) \text{ with } 1\le b\le 4 ,  r=2 \text{ if } b=1, \text{  and } 3 \le r \le 2b \text{ if } b\ge 2\\
                                                                                         \end{array}
                                         \right.$$
                                         when $c_1=(2,1)$.
\end{theorem}
Note that we have $c_2(\Ee)=(f_1, f_2)=(e_2, e_1)$ where $(e_1, e_2)$ is the bidegree of a curve associated to $\Ee$ via the Hartshorne-Serre correspondence (see Definition \ref{bidegree}). For cases $c_1\in \{(1,1), (a,0), (0,b)~|~ a,b\in \NN\}$ we also obtain similar classification of possible quadruples $(s;f_1, f_2;r)$ in Proposition \ref{ca4}, Proposition \ref{ca5} and Corollary \ref{qqq1}. These results give us a complete answer for the classification of globally generated vector bundles with the first Chern class $c_1$ with $(0,0) \le c_1<(2,2)$. We also completely describe the globally generated vector bundles with respect to quadruple $(s;f_1, f_2;r)$ in most cases.
In the rank $2$ case the list is so short that we may put it in the introduction.

\begin{theorem}\label{1.2intro}
Let $\Ee$ be a globally generated vector bundle of rank $2$ on $X=\PP^1\times \PP^2$ with no trivial factor and the first Chern class $c_1$ such that $(0,0)\le c_1 <(2,2)$. Then $\Ee$ is isomorphic to either
\begin{enumerate}
\item a direct sum of two line bundles, or
\item a Ulrich bundle arising from a non-trivial extension\\
$0\to\Oo_X(0,1)\to\Ee\to\Oo_X(2,0)\to 0,$
\item a twist of $\pi_2^*(T\PP^2(-1))$.
\end{enumerate}
\end{theorem}

So in the rank $2$ case we found no unexpected bundle, finitely many isomorphism classes of bundles and these isomorphic classes are distinguished by their $c_1$ and $c_2$. The situation is quite different in rank higher than $2$, although sometimes an expected bundle of high rank gives a full classification for lower ranks (see Proposition \ref{ccca2}).

Let us here summarize the structure of this paper. In Section $2$, we introduce the definitions and main properties that will be used throughout the paper, mainly the Hartshorne-Serre correspondence that relates the globally generated vector bundles with smooth curves contained in the Segre threefolds. In Section $3$ we collect several general results when one factor of bidegree of associated curve is $1$, and deal with the case of $c_1=(1,1)$. In Section $4$ and $5$ we keep classifying the globally generated vector bundles with $c_1=(2,1)$ and $c_1=(1,2)$ respectively.

\section{Preliminaries}
In this section we introduce the definitions and main properties of Segre varieties as well as the Hartshorne-Serre construction.

\begin{definition}
For $n_1, \ldots, n_s \in \NN$, let us define the Segre variety $X=\Sigma_{n_1, \ldots, n_s}$ to be the product $\PP^{n_1} \times \cdots \times \PP^{n_s}$ of $s$ projective spaces.
\end{definition}

Letting $\pi_i : \PP^{n_1} \times \cdots \times \PP^{n_s} \to \PP^{n_i}$ be the projection to the $i$-th factor, we denote $\pi_1^*\Oo_{\PP^{n_1}}(a_1) \otimes \cdots \otimes \pi_s^*\Oo_{\PP^{n_s}}(a_s)$ by $\Oo_X(a_1, \ldots, a_s)$. For our simplicity let us denote $\Oo_X(a, \ldots, a)$ by $\Oo_X(a)$, if there is no confusion. For a coherent sheaf $\Ee$ on $X$, we denote $\Ee \otimes \Oo_X(a_1, \ldots, a_s)$ by $\Ee(a_1, \ldots, a_s)$ and the dual sheaf of it by $\Ee^\vee$.

The intersection ring $A(X)$ is isomorphic to $A(\PP^{n_1}) \otimes \cdots \otimes A(\PP^{n_s})$ and so we have
\begin{equation}\label{int+}
A(X) \cong \ZZ [t_1, \ldots, t_s]/(t_1^{n_1+1}, \ldots, t_s^{n_s+1}).
\end{equation}
Then $X$ is embedded into $\PP^N$ with $N=\prod_{i=1}^s (n_i+1)-1$ by the complete linear system $|\Oo_X(1)|$. Note that $A^1(X)\cong \ZZ^{\oplus s}$ by $a_1t_1 + \cdots a_st_s \mapsto (a_1, \ldots, a_s)$.

Let $\Ee$ be a globally generated vector bundle of rank $r$ on $X$ with the first Chern class $c_1(\Ee)=(a_1, \ldots, a_s)$. Then it fits into the exact sequence
\begin{equation}\label{equ+}
0\to \Oo_X^{\oplus (r-1)} \to \Ee \to \Ii_C (a_1, \ldots, a_s) \to 0,
\end{equation}
where $C$ is a locally complete intersection of codimension $2$ on $X$ by \cite[Section $2$. $\mathbf{G}$]{man}. If $C$ is empty, then $\Ee$ is isomorphic to $\Oo_X^{\oplus (r-1)} \oplus \Oo_X(a_1, \ldots, a_s)$. In this article we will mainly work on $X=\PP^1 \times \PP^2$ and so the locally complete intersection $C$ would be a curve in $X$.

\begin{proposition}\cite{sierra}\label{prop1}
If $\Ee$ is a globally generated vector bundle of rank $r$ on $X$ with the first Chern class $c_1$ such that $H^0(\Ee(-c_1))\not= 0$, then we have $\Ee\simeq \Oo_X^{\oplus (r-1)}\oplus \Oo_X(c_1)$.
\end{proposition}

In particular, in the classification of globally generated vector bundles on $X$, we may assume that $C$ is not empty and $H^0(\Ee(-c_1))=0$.

Let $\hat{\pi}_i : X \to \PP^{n_1} \times \cdots \widehat{\PP^{n_i}} \times \cdots \times \PP^{n_s}$ denote the natural projection.

\begin{proposition}\label{LABEL}
Let $c_1=(a_1, \ldots, a_s)\in \ZZ_{\ge0}^{\oplus s}$ with $a_s=0$. Then there is a bijection given by $\Ee \mapsto \hat{\pi}_s^*(\Ee)$ between the set of spanned vector bundles $\Ee$ of rank $r$ on $\PP^{n_1} \times \cdots \times \PP^{n_{s-1}}$ with $c_1(\Ee)=(a_1, \ldots, , a_{s-1})$ and the spanned vector bundles of rank $r$ on $X$ with the first Chern class $c_1$. Moreover we have
\begin{enumerate}
\item $h^i (\PP^{n_1}\times \cdots \times \PP^{n_{s-1}}, \Ee) = h^i (X, \pi_{s}^* (\Ee))$ for all $i\ge 0$;
\item for any spanned bundle $\Gg$ on $X$ with $c_1(\Gg)=0$, we have $\Gg \cong \hat{\pi}_{s}^*(\hat{\pi}_{s*}(\Gg))$ with $\hat{\pi}_{s*}(\Gg)$ a spanned bundle on $\PP^{n_1} \times\cdots \times  \PP^{n_{s-1}}$.
\end{enumerate}
\end{proposition}
\begin{proof}
The last statement follows from the projection formula and the Leray spectral sequence of $\hat{\pi}_s$, because $\hat{\pi}_{s*}(\Oo _X) \cong \Oo _{\PP^{n_1}\times \cdots \times \PP^{n_{s-1}}}$ and $R^i\pi _{s\ast }(\Oo _X) =0$ for all $i\ge 0$. Fix a spanned vector bundle $\Gg$ of rank $r$ on $X$ with $c_1(\Gg )=0$ and let $\phi : X\to \mathbb{G}(r,m)$ be the associated morphism to a Grassmannian with $m=h^0(\Gg)$. To prove the first assertion it is sufficient to prove that $\phi$ is constant on every fiber of $\hat{\pi}_s$. Let $L \cong \PP^{n_s}$ be any fiber of $\hat{\pi} _s$ and then the vector bundle $\Gg_{\vert_L}$ is spanned with trivial determinant. Therefore $\Gg_{\vert_L}\cong \Oo _L^{\oplus r}$ and so $\phi (L)$ is a point. Since $\Gg \cong \hat{\pi}_{s\ast }(\Ee)$ for some $\Ee$ on $\PP^{n_1} \times\cdots \times  \PP^{n_{s-1}}$, the projection formula gives $\Ee \cong \hat{\pi}_{s*}(\Gg )$.
\end{proof}

Let us say that a line bundle $\Rr$ on $X$ has property $\diamond$ if the following conditions are satisfied:
\begin{enumerate}
\item $\Rr$ is globally generated;
\item $\Rr \ne \Oo _X$;
\item If $\Aa$ is a globally generated line bundle with $H^0(\Rr\otimes \Aa^\vee )\ne 0$,
then we have either $\Aa \cong \Oo _X$ or $\Aa\cong \Rr$.
\end{enumerate}

When we have $\dim (X)=3$, the Hartshorne-Serre correspondence can be stated as follows:

\begin{theorem}\cite[Theorem 1]{Arrondo}
Let $\Rr$ be a line bundle on $X$ with $\diamond$.
\begin{enumerate}
\item[(i)] There is a bijection between the set of pairs $(\Ee ,j)$, where $\Ee$ is a spanned vector bundle of rank $2$ on $X$ with $\det (\Ee )\cong \Rr$ and $j: \mathbb{K} \to H^0(\Ee)$ is non-zero map, up to linear automorphisms of $\mathbb{K}$, and the smooth curves $C\subset X$ with $\Ii _C \otimes \Rr$ spanned and $\omega _C \cong \Rr_{\vert _C}$, except that $C =\emptyset$ corresponds to $\Oo _X\oplus \Rr$ with a section $s$ nowhere vanishing.

\item[(ii)] Fix an integer $r\ge 3$. There is a bijection between the set of triples $(\Ee ,V,j)$, where $\Ee$ is a spanned vector bundle of rank $r$ on $X$ with $\det (\Ee )\cong \Rr$, $V$ is an $(r-1)$-dimensional vector space and $j: V \to H^0(\Ee)$ is a linear map, up to a linear automorphism of $V$, with dependency locus of codimension $2$, and the smooth curves $C\subset X$ with $\Ii _C \otimes \Rr$ spanned and $\omega _C \otimes \Oo_X\otimes \Rr^\vee$ spanned, except that $C =\emptyset$ corresponds to $\Oo _X^{\oplus (r-1)}\oplus \Rr$ with $V = H^0(\Oo _X^{\oplus (r-1)})$. $\Ee$ has no trivial factors if and only if $j$ is injective.

\item[(iii)] There is a bijection between the set of all pairs $(\Ff ,s)$, where $\Ff$ is a spanned reflexive sheaf of rank $2$ on $X$ with $\det (\Ee )\cong \Rr$ and $0\ne s\in H^0(\Ee )$, and reduced curves $C\subset X$ with $\Ii _C\otimes \Rr$ spanned and $\omega _C \otimes \Rr^\vee$ spanned outside finitely many points, except that $C =\emptyset$ corresponds to $\Oo _X\oplus \Rr$ with $s$ nowhere vanishing.
\end{enumerate}
\end{theorem}

\begin{example}\label{ooo0}
On a smooth and connected projective threefold $X$, let us fix a globally generated line bundle $\Ll$ with $h^0(\Ll )\ge 2$ and set $r:= h^0(\Ll )-1$. Since $\Ll$ is globally generated, so the evaluation map $\psi: H^0(\Ll)\otimes \Oo _X\to \Ll$ is surjective and $\ker (\psi )$ is a vector bundle of rank $r$ on $X$. The vector bundle $\Ff := \ker (\psi )^\vee$ fits in an exact sequence
\begin{equation}\label{oo1}
0 \to \Ll ^\vee \to \Oo _X^{\oplus (r+1)} \to \Ff \to 0
\end{equation}
and it determines the Chern classes of $\Ff$. If $h^1(\Ll ^\vee )=0$, e.g. $\Ll$ is ample, then the sequence (\ref{oo1}) gives $h^0(\Ff ) =r+1$. If $Y\subset X$ is the complete intersection of two elements of $|\Ll|$, then we get $Y$ as the dependency locus of a certain $(r-1)$-dimensional linear subspace of $H^0(\Ff)$.
\end{example}

\begin{remark}\label{24nov}
Assume that $C$ is a curve with $\omega _C \cong \Oo _C(c_1-c_1(X))$. If $C$ has $s$ connected components, then we have $h^0(\omega _C(c_1(X)-c_1)) =s$ and so the Hartshorne-Serre correspondence shows that $C$ gives a vector bundle $\Ee$ with $c_1(\Ee)=c_1$ of rank $r$ with no trivial factor if and only if $2 \le r \le s+1$.
\end{remark}

\begin{remark}\label{uuu00}
On a smooth threefold $X$, let us fix a very ample line bundle $\Ll$ and a smooth curve $C\subset X$. Assume that $\Ii _C\otimes \Ll$ is globally generated and take two general divisors $M_1,M_2\in |\Ii _C\otimes \Ll|$. Set $Y:= M_1\cap M_2$. By the Bertini theorem we have $Y =C\cup D$ with either
$D=\emptyset$ or $D$ a reduced curve containing no component of $C$ and smooth outside $C\cap D$.
\end{remark}

From now on let us fix $X$ to be the Segre threefold $\PP^1 \times \PP^2$, the product of a projective line and a projective plane. Let $V_1$ and $V_2$ be $2$ and $3$-dimensional vector spaces with the coordinates $[x_0, x_1]$ and $[y_0, y_1, y_2]$, respectively. Let $X\cong \PP (V_1) \times \PP (V_2) $ and then it is embedded into $\PP^5\cong \PP(V)$ by the Segre map where $V=V_1 \otimes V_2$.

The intersection ring $A(X)$ is isomorphic to $A(\PP^1) \otimes A(\PP^2) $ and so we have
$$A(X) \cong \ZZ[t_1, t_2]/(t_1^2, t_2^3).$$
We may identify $A^1(X)\cong \ZZ^{\oplus 2}$ by $a_1t_1+a_2t_2 \mapsto (a_1, a_2)$. Similarly we have $A^2(X) \cong \ZZ^{\oplus 2}$ by $b_1t_1t_2+b_2t_2^2\mapsto (b_1, b_2)$ and $A^3(X) \cong \ZZ$ by $ct_1t_2^2 \mapsto c$. Then $X$ is embedded into $\PP^5$ as a subvariety of degree $3$ by the complete linear system $|\Oo_X(1,1)|$, because $(t_1+t_2)^3=3t_1t_2^2$.

If $\Ee$ is a coherent sheaf of rank $r$ on $X$ with the Chern classes $c_1=(a_1, a_2)$, $c_2=(b_1, b_2)$ and $c_3$ we have:
\begin{align*}
c_1(\Ee (s_1, s_2))&=(a_1+rs_1, a_2+rs_2)\\
c_2(\Ee(s_1, s_2)) &=c_2+c_1(s_1, s_2)+(s_1, s_2)^2\\
\chi (\Ee)&=r+\frac{1}{2} \Big[ 2a_1+(a_1+1)a_2(a_2+3)\\
              &-(a_1+2, a_2+3)(b_1, b_2)+c_3\Big]
\end{align*}
for $(b_1, b_2)\in \ZZ^{\oplus 2}$.

Now Proposition \ref{LABEL} implies the following when $X=\PP^1 \times \PP^2$.
\begin{corollary}\label{qqq1}
We have the following one-to-one correspondences:
\begin{enumerate}
\item There is a bijection $\Ff \mapsto \pi_1^\ast (\Ff)$ between the spanned vector bundles of rank $r$ on $\PP ^1$ with degree $a$ and the spanned vector bundles of rank $r$ on $X$ with $c_1 = (a,0)$.
\item There is a bijection $\Ff\mapsto \pi _2^{\ast}(\Ff)$ between the spanned vector bundles of rank $r$ on $\PP^2$ with $c_1=b$ and the spanned vector bundles of rank $r$ on $X$ with $c_1 =(0,b)$.
\end{enumerate}
\end{corollary}

In particular we have a full classification of globally generated vector bundles on $X$ with $c_1=(1,0)$ or $(0,b)$ with $b\le 2$, since we have a full list of globally generated vector bundles on $\PP^2$ with $c_1\le 2$ in \cite{SU}. Thus we may assume $c_1 =(a,b)$ with $a,b>0$. We will mainly work on all the cases with $a+b\le 3$.


\section{Warm-up and case of $c_1=(1,1)$}
Let $\Ee$ be a globally generated vector bundle of rank $r\ge2$ on $X$ with the first Chern class $c_1=(a_1, a_2)$, $a_1\ge a_2 > 0$. Then it fits into the exact sequence
\begin{equation}\label{eqa1}
0\to \Oo_X^{\oplus (r-1)} \to \Ee \to \Ii_C(c_1) \to 0
\end{equation}
for a general $r-1$ sections of $\Ee$, where $C$ is a smooth curve. Write $C = C_1\sqcup \cdots \sqcup C_s$ with each $C_i$ connected.

If the rank is $r=2$, we have $\Ee^\vee \cong \Ee (-c_1)$ and $\omega_C\cong \Oo_C(c_1-(2,3))$. In this case we have $h^1(\Ee^\vee)=h^1(\Ii_C)$ since $h^1(\Oo_X(-c_1))=0$ by the K\"{u}nneth formula. Since $s$ is the number of the connected components of $C$, we have $h^1(\Ee^\vee)=s-1$. If $r\ge 3$, then we have $c_3(\Ee ) = \deg (\omega _C((2,3)-c_1))$ and $\omega_C((2,3)-c_1)$ is spanned.

\begin{remark}\label{rem1}
The intersection theory on $X$ is given by two variables $t_1, t_2$ with the relations $t_1^2=0$, $t_2^3=0$ and $t_1t^2=1$. Thus we have
$$(at_1+bt_2)(at_1+bt_2)(t_1+t_2) =  (2abt_1t_2+b^2t_2^2)(t_1+t_2) = 2ab+b^2.$$
\end{remark}

\begin{remark}\label{rem1.1.1}
Let $Y \subset X$ be the complete intersection of two divisors of type $|\Oo _X(a,b)|$ containing $C$. By Remark \ref{rem1} we have $\deg (Y)=2ab+b^2$, where $\deg(Y)$ is the degree of $Y$ as a curve in $\PP^5$. By the Bertini theorem $Y$ is a curve containing $C$ and smooth outside $C$. Note that $C$ occurs with multiplicity one in $Y$, because $\Ii _C(a,b)$ is spanned and so, affixing $P_i\in C_i$, $1\le i \le s$, we may find a divisor $T\in |\Ii _C(a,b)|$ not containing
the tangent line of $C_i$ at $P_i$. A Koszul complex standard exact sequence gives $h^0(\Oo _Y) =1$ and so $Y$ is connected. The adjunction formula gives $\omega _Y \cong \Oo _Y(2a-2,2b-3)$ and so we have
\begin{align*}
2p_a(Y) -2& = (at_1+bt_2)(at_1+bt_2)((2a-2)t_1+(2b-3)t_2) \\
&= (-3+2b)2ab + b^2(-2+2a)\\
&=2b(3ab-3a-b).
\end{align*}
\end{remark}

\begin{example}\label{rem1.1.2}
Let $Y\subset X$ be any complete intersection of two elements of $|\Oo _X(a,b)|$ with $(a,b)\in \NN^2\setminus \{(0,0)\}$. We have $\omega _Y\cong \Oo _Y(2a-2,2b-3)$ and hence $\omega _Y(2-a, 3-b) \cong \Oo _Y(a,b)$. Therefore $Y$ is not the zero-locus of a section of a bundle of rank $2$. Note that $h^0(\omega _Y(2-a,3-2b)) = (a+1)\binom{b+2}{2} -2$. As in Example \ref{ooo0} we see that $Y$ gives a globally generated bundle with $c_1=(a,b)$ and $Y$ is a dependency locus for a bundle of rank $r$ if and only if $3\le r \le (a+1)\binom{b+2}{2}-1$. If $r= (a+1)\binom{b+2}{2}-1$, there is a unique bundle $\Ee$ associated to some complete intersection curve. Since $\Ee$ is unique, we have $g^\ast (\Ee )\cong \Ee$ for each $g\in \mathrm{Aut} (X) = \mathrm{Aut} (\PP^1)\times \mathrm{Aut} (\PP^2)$, i.e. $\Ee$ is homogeneous.
\end{example}

Let $T$ be an integral projective curve. By the universal property of the projective spaces, there is a bijection between the set of morphisms $u: T\to X$ and the set of pairs $(u_1,u_2)$ of morphisms $u_1: T \to \PP^1$ and $u_2: T \to \PP^2$. If $u$ is an embedding, then we have
$$\deg (\Oo _{u(T)}(a,b)) = a\deg (u_1)+b\deg (u_2)\deg (u_2(T))$$
with the convention that $\deg (u_2) =\deg (u_2(T)) =0$ if $u_2(T)$ is a point.

Now assume that $u_2(T)$ is a curve and $T$ is smooth. Let $u'': J \to u_2(T)$ denote the normalization map and then there is a unique morphism $u': T\to J$ such that $u_2 = u'\circ u''$. We have $\deg (u_2) =\deg (u')$.

Now take as $T$ a connected component $C_i$ of $C$. Write
$$e[i]_1:= \deg (u_1) ~~\text{ and }~~e[i]_2 = \deg (u_2) \deg (u_2(C_i)).$$
If $e[i]_2 \ne 0$, set $e'_2:= \deg (u')$ and $e''_2:= \deg (u_2(T))$. If $e[i]_2=0$, set $e'_2=e''_2=0$.

\begin{definition}\label{bidegree}
For a curve $C\subset X$, the {\it bidegree} of $C$ is defined to be the pair of integers $(e_1, e_2)$ with
$$(e_1, e_2):=\left(\sum_{i=1}^s e[i]_1, \sum_{i=1}^s e[i]_2\right).$$
Indeed, we have $e_1=\deg (\Oo_C(1,0))$ and $e_2=\deg (\Oo_C(0,1))$.
\end{definition}

\begin{lemma}\label{ca3}
Assume $c_1 =(a,1)$ with $a\ge 0$.
\begin{itemize}
\item [(i)] We have $e[i]_1 \in \{0,1\}$ for each $i\in \{1,\dots ,s\}$.
\item[(ii)] If $e[i]_1 =0$, then $e[i]_2=1$ and there are $P_i\in \PP^1$ and a line $L_i\subset \PP^2$ such that $C_i = \{P_i\}\times L_i$.
\item[(iii)] If $e[i]_1=1$ for some $i$, then $s=1$ and $C$ is rational.
\end{itemize}
\end{lemma}

\begin{proof}
If $e[i]_1 =0$, then there are $P_i\in \PP^1$ and $L_i\in |\Oo _{\PP^2}(e[i]_2)|$ such that $C_i = \{P_i\}\times L_i$. Since $\Ii _C(a,1)$ is globally generated, we get $e[i]_2=1$.

Assume $e[i]_1\ne 0$ and then $\pi _1(C_i) =\PP^1$. For a fixed point $P\in \PP^1$, the degree of ${\Oo _X(a,1)}_{\vert {\{P\}\times \PP^2}}$ is $1$. Since $\Ii _C(a,1)$ is globally generated, so we get $\deg (C\cap \pi _1^{-1}( P)) =1$, i.e. $e[i]_1=1$. We also see that ${\pi _1}_{\vert{C_i}}: C_i\to \PP^1$ has degree $1$ and so $C_i$ is rational. If $e_1\ne 0$, we get $\deg ({\pi _1}_{\vert C}) =1$ and so there is a unique index $i$ with $e[i]_1=1$. Assume $s\ge 2$ and fix $j\in \{1,\dots ,s\}\setminus \{i\}$. Since $e[j]_1=0$, there is $P_j\in \PP^1$ and a line $L_j\subset \PP^2$ with $C_j =\{P_j\}\times  L_j$. Set $\{P_j,O\}:= C_i\cap \{P_j\}\times \PP^2$. Since $C_i\cap C_j= \emptyset$, we have $O\notin L_j$ and so $\{P_j\}\times \PP^2$ is in the base locus of $\Ii _C(1,1)$.
\end{proof}

\begin{lemma}\label{ca3.1}
Assume $c_1=(1,b)$ with $b\ge 0$.
\begin{itemize}
\item [(i)] ${\pi _2}_{|C}: C \to \PP^2$ is an embedding.
\item [(ii)] $e_1>0$ and $C$ is isomorphic to a smooth plane curve of degree $e_2$.
\end{itemize}
\end{lemma}

\begin{proof}
Let us fix a point $P\in \PP^2$ and set $T:= \PP^1\times \{P\}$. Since $\Oo _T(1,3-b) \cong \Oo _T(1,0)$ has degree $1$, so $\omega _T(1,3-b)$ has negative degree and in particular it is not spanned. Thus $T$ is not a connected component of $C$. Since $\Ii _C(1,b)$ is globally generated, we get $\deg (T\cap C) \le 1$ for all $P\in \PP^2$ and so ${\pi _2}_{\vert_C}: C\to \PP^2$ is an embedding. Since each plane curve is connected, we get $s=1$. Since ${\pi _2}_{|_C}$ is an embedding, we also get that $C$ is isomorphic to a smooth plane curve of degree  $e_2$.
\end{proof}

\begin{proposition}\label{ca4}
Let $\Ee$ be a globally generated vector bundle of rank $2$ with $c_1(\Ee )=(1,1)$ and no trivial factor. Then we have $\Ee \cong \Oo _X(1,0)\oplus \Oo _X(0,1)$.
\end{proposition}

\begin{proof}
The zero-locus of a general section of $\Oo _X(1,0)\oplus \Oo _X(0,1)$ has bidegree $(0,1)$, because $c_2(\Oo _X(1,0)\oplus \Oo _X(0,1)) =t_1t_2$. Since $\Ii_C(1,1)$ is spanned, so $C$ is contained in the complete intersection $Y$ of two hypersurfaces in $|\Ii_C(1,1)|$. Note that we have $\deg (Y)=3$ and $\omega_Y \cong \Oo_Y(0,-1)$; we have $2p_a(Y)-2 = (t_1+t_2)(t_1+t_2)(-t_2) =-2$ and so $p_a(Y)=0$.  Thus we have $\deg ( C)\le 3$ with the equality if and only if $C = Y$ is a linear section of codimension $2$.

The case of $\deg (C)=3$ is not possible since $\omega_C \cong \Oo_C(-1,-2)$. Now assume $\deg ( C)\leq 2$. Since $c_2(\Oo _X(1,0)\oplus \Oo _X(0,1)) = t_1t_2$, the curves associated to $\Oo _X(1,0)\oplus \Oo _X(0,1)$ have
bidegree $(0,1)$. Since $\omega _C \cong \Oo _C(-1,-2)$ and $\Oo _X(1,2)$ is ample,
each $C_i$ is rational and $e[i]_1+2e[i]_2 =-\deg (\omega _{C_i})=2$. Thus each connected component $C_i$ of $C$ is a curve of bidegree $(0,1)$. Lemma \ref{ca3.1} gives $s=1$. Since $h^0(\Oo _C(0,1)) =2 < h^0(\Oo _X(0,1))$, we have
$h^0(\Ii _C(0,1)) >0$ and so any associated bundle $\Ee$ has a non-zero map $f: \Oo _X(1,0)\to \Ee$ because of $h^1(\Oo _X(-1,0)) =0$. Since $h^0(\Ii _C(-1,1)) =h^0(\Ii _C(0,0))=0$, so $\textrm{coker}(f)$ is torsion-free, i.e. $\textrm{coker}(f) \cong \Ii _T(0,1)$ with either $T =\emptyset$ or $T$ a locally complete intersection curve. Since $c_2(\Oo _X(1,0)\oplus \Oo _X(0,1)) =t_1t_2$ and $C$ has bidegree $(0,1)$, so we have $T=\emptyset$. It implies $\Ee \cong \Oo _X(1,0)\oplus \Oo _X(0,1)${\color{blue},} since $h^1(\Oo_X(1,-1))=0$.
\end{proof}

\begin{proposition}\label{ca5}
Let $\Ee$ be a globally generated vector bundle of rank $r\ge 3$ on $X$ with $c_1=(1,1)$ and no trivial factor. Then we have $r\in \{3,4,5\}$ and $\Ee$ is one of the following:
\begin{itemize}
\item [(i)] $0\to \Oo_X(-1,-1) \to \Oo_X^{\oplus (r+1)} \to \Ee \to 0$ with $r\in\{3,4,5\}$,\\
 the associated curve is the complete intersection $Y$ of two hypersurfaces in $|\Oo_X(1,1)|$;
\item [(ii)] $\Ee \cong \Oo _X(1,0)\oplus \pi _2^*(T\PP^2(-1))$.
\end{itemize}
\end{proposition}

\begin{proof}
If the associated curve is the complete intersection $Y$ of two hypersurfaces in $|\Oo_X(1,1)|$ we are in the case $\mathrm{(i)}$ by Example \ref{rem1.1.2} and so we may assume that the associated curve $C$ is not a complete intersection.

Since in the case $r=2$ no bundle has a section with $s>1$ by Proposition \ref{ca4}, we only need to check the smooth curves $C$ with $\Ii _C(1,1)$ globally generated and with $\omega_C(1,1)$ globally generated, but not trivial.
 It only remains to check the cases with $s=1$ and $\deg ( C)=2$. $C$ cannot have bidegree $(2,0)$, because it is irreducible.

Now take a conic with bidegree $(0,2)$. There are $P\in  \PP^1$ and $C'\in |\Oo _{\PP^2}(2)|$ with $C = \{P\}\times C'$. Then $\{P\}\times \PP^2$ is in the base locus of $\Ii _C(1,1)$.

Hence the only bundles not associated to the complete intersection curve may come from a connected $C$ with bidegree $(1,1)$. For any such a curve $C$ we have $\deg (\omega _C(1,2)) =1$ and $h^0(\omega _C(1,2)) =2$. Thus it gives a bundle if and only if $r=3$. Any connected curve $C$ with bidegree $(1,1)$ is associated to a pair $(u_1,u_2)$ with $u_1: \PP^1 \to \PP^1$ an isomorphism and $u_2: \PP^1\to \PP^2$ an embedding as a line. Hence for any two such curves $C$ and $C'$, there is $f\in \mathrm{Aut}(X)$ with $f^\ast ({C}) =C'$. The bundle $\Ee_0:=\Oo _X(1,0)\oplus \pi _2^*(T\PP^2(-1))$ has the Chern polynomial $(1+t_1)(1+t_2+t_2^2) = 1+ t_1+ t_2+ t_1t_2+ t_2^2+ t_1t_2^2$ and so it is associated to a connected curve with bidegree $(1,1)$. For any $f\in \mathrm{Aut}(X)$ we have $f^* \Ee_0 \cong \Ee_0$ and so $\Ee_0$ is the only bundle associated to a curve of bidegree $(1,1)$.
\end{proof}

\begin{remark}
If $r=5$ and $C=Y$, then we have $\Ee \cong T\PP^5(-1)_{\vert_X}$, because Remark \ref{ooo0} gives the uniqueness of such a bundle. We denote by $S_r$ the family of bundles obtained from $\mathrm{(i)}$. We may compute $\dim(S_r)=h^0(\Oo_X^{\oplus (r+1)}(1,1))-(r+1)^2$ and so we get 
$$\dim(S_5)=0~~,~~\dim(S_4)=5~~,~~\dim(S_3)=8.$$
In fact $S_4$ is given by the pull-backs of  $T\PP^4(-1)$ by a projection from a point of $\PP^5\setminus X$ and $S_3$ is given by the pull-backs of  $T\PP^3(-1)$ by a projection from a line of $\PP^5$ not meeting $X$.
\end{remark}


\section{Case of $c_1=(2,1)$}
Let $Y$ be any complete intersection of two hypersurfaces in $|\Oo _X(2,1)|$. Since $(2t_1+t_2)(2t_1+t_2)t_1 = 1$ and $(2t_1+t_2)(2t_1+t_2)t_2 = 4$, we have
$\deg (\Oo _Y(1,0)) =1$ and $\deg (\Oo _Y(0,1)) =4$. In particular the connected curve $Y\subset \PP^5$ has degree $5$. We have $\omega _Y \cong \Oo _Y(2,-1)$. Since $(2t_1+t_2)(2t_1+t_2)(2t_1-t_2)= -2$, we have $p_a(Y)=0$. It implies that the sectional genus of $\Oo_X(2,1)$ is zero. Notice $\omega _Y\not \cong \Oo _Y(0,-2)$, because $\deg (\Oo _Y(0,-2)) =-4$. Example \ref{rem1.1.2} shows that $Y$ gives a globally generated vector bundle of rank $r$ if and only if $3 \le r \le 8$.

\begin{lemma}\label{cb1}
Assume that $\Ii _C(2,1)$ is globally generated. Then one of the following cases occurs:
\begin{itemize}
\item[(i)] $(s;e_1, e_2)=(1;1, b)$ with $1\le b\le 4$.
\item[(ii)] $(s; e_1, e_2)=(s;0,s)$ with $1\le s \le 3$. \\In particular there are $s$ distinct points $P_i\in \PP^1$ and lines $L_i \subset \PP^2$ such that $C_i = \{P_i\}\times L_i$ for each $i$; in this case $\omega _C \cong \Oo _C(0,-2)$.
\end{itemize}
\end{lemma}

\begin{proof}
Since $Y$ has bidegree $(1,4)$, we have $e_1\le 1$ and $e_2\le 4$. If $e_1 >0$, we have $s=1$ and $1\le e_2\le 4$ by Lemma \ref{ca3}.

Now assume $e_1=0$. By Lemma \ref{ca3} we have $s = e_2$ and there are $P_i\in \PP^1$
and a line $L_i\subset \PP^2$ such that $C_i = \{P_i\}\times L_i$. If $(e_1,e_2) =(0,4)$, then $Y =C\cup D$ with $D$ a line of type $(1,0)$ since $Y$ has bidegree $(1,4)$. Since $h^0(\Oo _Y) =1$ (see Remark \ref{rem1.1.1}), we get $\deg (D\cap C)\ge 4$. Therefore $\Ii _C(2,1)$ is not globally generated.
\end{proof}

\begin{remark}\label{cb2}
We do not claim that each $C$ listed in Lemma \ref{cb1} has $\Ii _C(2,1)$ globally generated. The case $(e_1,e_2) =(1,4)$ corresponds to the complete intersection curve $Y$.
We check in Example \ref{cb7} that the case $e_2\in \{2,3\}$ is realized. Since in both cases $C$ is connected with arithmetic genus zero, we have $h^0(\omega _C(0,2)) = 2e_2-1$ and $\omega _C(0,2)$ is spanned (and trivial if $e_2=1$). Hence the existence part for the case
$(s;e_1,e_2)=(1;1,1)$ gives globally generated bundles with $s=1$ and bidegree $(1,1)$ if and only if $r=2$, while the existence part for $(s;e_1,e_2) = (1;1,b)$ with $2\le b\le 4$,
gives bundles if and only if $3\le r \le 2b$.
\end{remark}

\begin{lemma}\label{cb3}
Consider a curve $C$ with $(s;e_1,e_2) =(2;0,2)$. Then the associated bundle of rank $r$ with no trivial factor is one of the following:
\begin{itemize}
\item [(i)] an extension of $\Oo_X(2,0)$ by $\Oo_X(0,1)$ if $r=2$; either $\Oo_X(2,0)\oplus \Oo_X(0,1)$ or one of the bundles in the $2$-dimensional family of Ulrich bundles in \cite[Theorem 4.1]{CMP};
\item [(ii)] $\Oo _X(0,1)\oplus \Oo _X(1,0)^{\oplus 2}$ if $r=3$.
\end{itemize}
In particular such a curve $C$ does not give bundles of rank higher than $3$ without trivial factors.
\end{lemma}

\begin{proof}
Each connected component of $C$ is a line of bidegree $(0,1)$. Since $\omega _C \cong \Oo _C(0,-2)$ and $s=2$, the associated bundles have rank $r\in \{2,3\}$ by Remark \ref{24nov}. In all cases we have $c_3(\Ee )=0$.

\quad (a) First assume $r=2$.  Since $(s;e_1,e_2) =(2;0,2)$, so there are $P_1, P_2\in \PP^1$ and lines $L_1,L_2\subset \PP^2$ such that $C = \{P_1\}\times L_1\cup \{P_2\}\times L_2$. In particular, we have $h^0(\Ii _C(2,0)) >0$ and so there is a non-zero map $f: \Oo _X(0,1)\to \Ee$. Since $L_1\cap L_2\ne \emptyset$, we have $P_1\ne P_2$ and so $h^0(\Ii _C(1,0)) =0$. Since $h^0(\Ii _C(2,-1)) =0$, so $\mathrm{coker}(f)$ is torsion-free, i.e. $\mathrm{coker}(f) =\Ii _T(2,0)$ with either $T=\emptyset$ or $T$ a locally complete intersection curve. Since $c_1(\Oo _X(2,0))\cdot c_1(\Oo _X(0,1)) =2t_1t_2$ and $C$ has bidegree $(0,2)$, we have $T = \emptyset$. Since $h^1(\Oo _X(-2,1)) =3$, we get non-trivial extensions of $\Oo _X(2,0)$ by $\Oo _X(0,1)$. Since the
two line bundles are both Ulrich, any extension of them is again Ulrich (see \cite[Theorem 4.1]{CMP}).

\quad (b) Assume $r=3$. Let $\Ff$ be the cokernel of a general map $u: \Oo _X\to \Ee$ and $v: \Ee \to \Ff$ denote the quotient map. Since $c_3(\Ee ) =0$, so $\Ff$ is locally free and it is as described in (i). In particular $\Oo _X(0,1)$ is a subbundle of $\Ff$ with $\Oo _X(2,0)$ as its cokernel.
Set $\Gg := v^{-1}(\Oo _X(0,1))$. Since $\Gg$ is an extension of $\Oo _X(0,1)$ by $\Oo _X$, we have $\Gg \cong \Oo _X\oplus \Oo _X(0,1)$. Therefore $\Ee$ is an extension of $\Oo _X(2,0)$ by $\Oo _X\oplus \Oo _X(0,1)$. In particular $\Oo _X(0,1)$ is a subbundle of $\Ee$. Set $\Hh := \Ee /\Oo_X(0,1)$. $\Hh$ is a globally generated bundle of rank $2$ on $X$ with $c_1(\Hh ) =(2,0)$. Since $\Ee$ has no trivial factor and $\Hh$ is a quotient of $\Ee$, so $\Hh$ has no trivial factor. By Proposition \ref{LABEL} we have $\Hh \cong \Oo _X(1,0)^{\oplus 2}$. Hence $\Ee \cong \Oo _X(0,1)\oplus \Oo _X(1,0)^{\oplus 2}$.
\end{proof}

\begin{lemma}\label{cb5}
For a curve $C$ with $(s;e_1,e_2) =(1;0,1)$, the associated globally generated bundle with no trivial factor is $\Oo _X(1,1)\oplus \Oo _X(1,0)$. In particular it has rank $2$.
\end{lemma}

\begin{proof}
Since $s=1$, we only get a bundle of rank $2$. We have $c_1(\Oo _X(1,1))\cdot c_1( \Oo _X(1,0)) = (t_1+t_2)t_1 =t_1t_2$ and so any curve associated to $c_1(\Oo _X(1,1))\cdot c_1( \Oo _X(1,0))$ has bidegree $(0,1)$. Now assume that $C$ has bidegree $(0,1)$ and then $s=1$. Let $\Ee$ be any associated bundle with no trivial factor. Lemma \ref{cb3} gives that $\Ee$ has rank $2$
and that there are $P\in \PP^1$ and a line $L\subset \PP^2$ such that $C = \{P\}\times L$. Since $h^0(\Ii _C(1,0)) =1$ and $h^0(\Ii _C(-1,1)) =h^0(\Ii _C(0,0)) =0$, there is a non-zero map $f: \Oo _X(1,1)\to \Ee$ with $\mathrm{coker} (f)$ torsion free, i.e. $\mathrm{coker} (f) \cong \Ii _T(1,0)$ with either $T=\emptyset$ or $T$ a locally complete intersection curve. Since $c_1(\Oo _X(1,1))\cdot c_1( \Oo _X(1,0)) = (t_1+t_2)t_1 =t_1t_2$, we get $T=\emptyset$ and so $\Ee \cong \Oo _X(1,1)\oplus \Oo _X(1,0)$.
\end{proof}

\begin{lemma}\label{cb6}
We have a curve $C$ with $(s;e_1,e_2) =(1;1,1)$ if and only if the associated globally generated bundle with no trivial factor is isomorphic to $\pi _2^*(T\PP^2(-1))(1,0)$.
\end{lemma}

\begin{proof}
Note that the bundle $\Ff := \pi _2^*(T\PP^2(-1))(1,0)$ is globally generated, $c_1(\Ff(1,0)) =(2,1)$ and $h^0(\Ff (-1,0)) = 3$. Since $h^1(\Oo _X(-1,0)) =0$, the associated smooth curve $C'$ satisfies $h^0(\Ii _{C'}(1,1))=3$. Thus $C'\subset \PP^5$ spans a plane and so it is connected, but not a line. Since $C'$ is smooth, connected and rational, so we get that $C'$ has $(s;e_1,e_2) =(1;1,1)$.

Now take any smooth $C$ with $(s;e_1,e_2) =(1;1,1)$ and let $\Ee$ be the associated bundle. $\omega _C \cong \Oo _C(0,-2)$ implies that $\Ee$ has rank $2$. Since $C$ is a smooth conic, we have $h^0(\Ii _C(1,1)) =3$ and so $h^0(\Ee (-1,0)) = 3$, because of $h^1(\Oo _X(-1,0)) =0$. Since $c_1(\Ee (-1,0)) =0$, to prove the lemma it is sufficient to prove that $\Ee (0,-1)$ is globally generated.

Let us fix a non-zero map $f: \Oo _X \to \Ee (-1,0)$. Since $h^0(\Ii _C(0,1)) = h^0(\Ii _C(1,0)) =0$, so $f$ has torsion-free cokernel, i.e. $\mathrm{coker} (f) \cong \Ii _T(0,1)$ with either $T= \emptyset$ or $T$ a locally complete intersection curve. Since $h^0(\Ee (-1,0)) = 3$, we get that $T$ is a line of bidegree $(0,1)$ and so $\Ii _T(0,1)$ is globally generated. $h^1(\Oo _X)=0$ implies that $\Ee (-1,0)$ is globally generated and so $\Ee (-1,0) \cong \pi _2^*(T\PP^2(-1))(1,0)$ by Proposition \ref{qqq1}.
\end{proof}

\begin{lemma}\label{cb6.1}
For a curve $C$ with $(s;e_1,e_2) = (1;1,2)$, its associated globally generated bundle with no trivial factor is one of the following:
\begin{itemize}
\item [(i)] $\Oo _X(2,0)\oplus \pi _2^\ast (T\PP^2(-1))$ ; $r=3$,
\item [(ii)] $\Oo _X(1,0)\oplus \Oo _X(1,0) \oplus  \pi _2^\ast (T\PP^2(-1))$ ; $r=4$.
\end{itemize}
\end{lemma}

\begin{proof}
Note that
\begin{align*}
(1+t_1)(1+t_1)(1+t_2+t_2^2)& =(1+2t_1)(1+t_2+t_2^2)\\
&=1 +(2t_1+t_2)+(2t_1t_2+t_2^2)+t_1t_2^2.
\end{align*}
Thus both bundles in the assertion are associated to a curve with bidegree $(1,1)$. By part (iii) of Lemma \ref{ca3}, we have $s=1$ in both cases. We have $C\cong \PP^1$, $\deg (\omega _C(0,2)) =2$, $h^0(\omega _C(0,2)) =3$ and hence $C$ is associated to a bundle of rank $r$ if and only $3\le r \le 4$. $C$ is associated to $(u_1,u_2)$, where $u_1: \PP^1\to \PP^1$ is an isomorphism and either $u_2: \PP^1\to \PP^2$ is an embedding as a smooth conic
or $\deg (u_2) =2$ and $u_2(\PP ^1)$ is a line.

Assume that the latter case occurs and call $L:= u_2(\PP ^1)\subset \PP^2$ the associated line. Let $\Gg$ be the rank $2$ reflexive sheaf fitting in an exact sequence
$$0 \to \Oo _X \to \Gg \to \Ii _C(1,2)\to 0.$$
We have $c_3(\Gg ) =\deg (\omega _C(0,2)) = 2$. Since $C \subset \PP^1\times L$, we have $h^0(\Ii _C(1,1)) = 2$. Since $h^1(\Oo _X(0,-1)) =0$, we get $h^0(\Gg (0,-1)) =2$. Since $h^0(\Ii _C(0,1)) =h^0(\Ii _C(0,1)) =0$,
a non-zero section of $\Gg (0,-1)$ induces an exact sequence
$$0\to \Oo _X\to \Gg (0,-1) \to \Ii _T(1,0)\to 0$$with $T$ a locally Cohen-Macaulay curve and $h^0(\Ii _T(1,0)) =1$. We get that $T$ is a line of bidegree $(0,1)$
and hence $c_3(\Gg ) = c_3(\Gg (0,-1)) =\deg (\omega _T(0,2)) =0$, a contradiction. Hence $u_2$ is always an embedding.

For any two $C$ and $C'$ as above, there is $f\in \mathrm{Aut}(X)$ such that $f^*({C}) \cong C'$. Since both bundles in the assertion are invariant after taking pull-back $f^*$, we get the lemma.
\end{proof}

Any irreducible element of $|\Oo _X(2,1)|$ is an irreducible surface of degree $4$ spanning $\PP^5$, i.e. it is a minimal degree surface in $\PP^5$. Inspired by the classification of minimal degree surfaces we make the following construction which proves the existence part for $s=1$ and the bidegrees $(1,b)$, $1\le b\le 3$.

\begin{example}\label{cb7}
Set $F_0:= \PP^1\times \PP^1$ and take $h\in |\Oo _{\PP^1\times \PP^1}(1,0)|$ and $f\in |\Oo _{\PP^1\times \PP^1}(0,1)|$. Let $\WW \subset \PP^5$ be the image of $F_0$ by
the very ample linear system $|h+2f|$. Therefore $f\in |\Oo _{\WW}(1,0)|$, $|h+f| = |\Oo _{\WW}(0,1)|$ and $|h|= |\Oo _{\WW}(-1,1)|$. Each element of $ |h|$ is a conic of $\PP^5$, while each element of $|f|$ is a line of $\PP^5$. Take $D\in |ah+bf|$ with $a\ge 0$ and $b \ge 0$. The line bundle $\Oo _{\WW}(h+3f)(-D)$ is globally generated if and only if $0\le a \le 1$ and $0\le b\le 3$. $D$ is a connected smooth rational curve if and only if either $a=1$ or $b=1$.

Now we prove that $\WW$ is contained in $X$ (it would then be an element of $|\Oo _X(2,1)|$, because $\deg (\WW) =4$ and $|h+3f|$ and $|3h+f|$ are the only complete linear systems on $\WW$ associated to an ample line bundle $\Ll$ with $h^0(\Ll )=8$). We need to find two morphisms $f_1: \WW \to \PP^1$ and $f_2: \WW \to \PP^2$ such that $(f_1,f_2)$ is an embedding. We saw that we need to take $f_1$ associated to the complete linear system $|f|$ and $f_2$ associated to a linear subspace of codimension $1$ in the complete linear system $|h+f|$. Let $V\subset H^0(\Oo _{\WW}(h+f))$ be a general hyperplane. Since $\dim (V) = \dim (\WW )+1$, it induces a morphism $f_2: \WW \to \PP^2$ with $\deg (f_2)=2$. The morphism $\phi = (f_1,f_2)$ has the image contained in $X$. We only need to check that for a general $V$ the morphism $\phi$ is an embedding. We first check that $\deg (\phi ) =1$. Indeed, since $\deg (f_2) =2$, it would be sufficient to observe that for general $V$ some points $P_1,P_2\in \WW$ with $f_2(P_1) =f_2(P_2)$, $\{P_1,P_2\}$ is not contained in a ruling of $F_0$. The complete linear system $|h+f|$ induces an embedding of $\WW$ into a smooth quadric surface in some $\PP^3$ and we may take as $V$ the hyperplane associated to the linear projection from a general point of $\PP^3$. Since $\deg (\phi )=1$, $\phi (\WW )$ is a non-degenerate surface of degree $4$ embedded in $X$. Now we use the classification of minimal degree non-degenerate integral surfaces of $\PP^5$. Since $f_2$ is a finite morphism, $\phi$ is finite. We know that $\phi (\WW )$ spans $\PP^5$ from $\deg (X)=3$. Since $\phi (\WW)$ has not $\ZZ$ as its Picard group, $\phi (\WW)$ is not the Veronese surface. Since $\phi (\WW ) \subset X$, either $\phi (\WW ) \cong \WW$ or $F_2$ (the other minimal degree smooth surface of $\PP^5$) or the cone over the rational normal curve of $\PP^5$, up to a projective equivalence. In particular $\phi (\WW )$ is a normal surface. Since $\phi (\WW )$ is normal and $\phi$ is finite and birational onto its image, $\phi : \WW \to \phi (\WW )$ is an isomorphism by the Zariski's Main Theorem. Fix
$b\in \{1,2,3\}$ and let $C\subset \WW$ be any smooth element of $|h+bf|$. We saw that $\Ii _{C,\WW}(h+3f) \cong \Oo _{\WW}((3-b)f)$ is globally generated. Since $\WW\in |\Oo _X(2,1)|$
and $\Oo _{\WW}(2,1) = \Oo _{\WW}(h+3f)$, we get that $C$ is globally generated.
\end{example}

Hence we proved the following result.

\begin{theorem}\label{ttt2}
There is a globally generated vector bundle $\Ee$ of rank $r\ge 2$ on $X$ with $c_1(\Ee )=(2,1)$, no trivial factor and associated curve with $(s;e_1,e_2)$ as associated data if and only if the quadruple $(s;e_1,e_2;r)$ are in the following list:
\begin{enumerate}
\item $(s;0,s;r)$ with $1\le s \le 3$ ; $2 \le r \le s+1$,
\item $(1;1,b;r)$ with $1\le b\le 4$ ; $r=2$ if $b=1$, and $3 \le r \le 2b$ if $b\ge 2$.
\end{enumerate}
Indeed we have
\begin{itemize}
\item [(i)] $(s;0,s;r) =(1;0,1;2) \Leftrightarrow \Ee \cong \Oo _X(1,1)\oplus \Oo _X(1,0)$.
\item [(ii)] $(s;0,s;r) =(2;0,2;2) \Leftrightarrow \Ee \cong \Oo _X(2,0)\oplus \Oo _X(0,1)$ or $\Ee$ is an Ulrich bundle arising from the non-trivial extensions\\
$0\to\Oo_X(0,1)\to\Ee\to\Oo_X(2,0)\to 0,$
\item [(iii)] $(s;0,s;r) =(2;0,2;3)\Leftrightarrow \Ee \cong \Oo _X(1,0)^{\oplus 2}\oplus \Oo _X(0,1)$.
\item [(iv)] $(s;e_1,e_2;r) =(1;1,1;2) \Leftrightarrow \Ee \cong \pi _2^*(T\PP^2(-1))(1,0)$.
\item [(v)] $(s;e_1,e_2;r) =(1;1,2;3) \Leftrightarrow \Ee \cong \Oo _X(2,0)\oplus \pi _2^*(T\PP^2(-1))$.
\item [(vi)] $(s;e_1,e_2;r) =(1;1,2;4) \Leftrightarrow \Ee \cong \Oo _X(1,0)\oplus \Oo _X(1,2)\oplus \pi _2^*(T\PP^2(-1))$.
\end{itemize}
\end{theorem}


\section{Case of $c_1=(1,2)$}
Let $\Ee$ be a globally generated vector bundle of rank $r \ge 2$ on $X$ with $c_1(\Ee)=(1,2)$ and then it fits into the sequence (\ref{eqa1}) with $c_1=(1,2)$. Then $C$ is contained in a complete intersection $Y$ of two hypersurfaces in $\Oo_X(1,2)$. Since $(t_1+2t_2)(t_1+2t_2) = 4t_1t_2 +4t_2^2$, the curve $Y$ has bidegree $(4,4)$ and degree $8$. We have $\omega _Y \cong \Oo _Y(0,1)$ and so $2p_a(Y)-2 = (t_1+2t_2)(t_1+2t_2)t_2 = 4$, i.e. $Y$ has genus $3$. Since $\omega _Y(1,1)$ has degree $12$, then $h^0(\omega _Y(1,1)) = 10$ and so the curve $Y$ gives a spanned bundle for all ranks $r$ with $3\le r \le 11$ (see Example \ref{rem1.1.2}). Lemma \ref{ca3.1} gives $s=1$. i.e. the smooth curve $C$ of bidegree $(e_1,e_2)$ is connected.

\begin{remark}\label{cca0}
 Since $Y$ has bidegree $(4,4)$, for any $C$ we have $e_1\le 4$ and $e_2\le 4$. Let $g$ be the genus of the connected curve $C$. Since each complete intersection of two ample divisors is connected, we get $g\le 3$. Lemma \ref{ca3.1} gives that $g\in \{0,1,3\}$ and
\begin{itemize}
\item [(i)] $g=3 \Leftrightarrow e_2=4$;
\item [(ii)] $g=1 \Leftrightarrow e_2=3$;
\item [(iii)] $g=0 \Leftrightarrow e_2 \in \{1,2\}$.
\end{itemize}
\end{remark}

\begin{lemma}\label{cca1}
For spanned bundles of rank $2$, we have the following:
\begin{enumerate}
\item $\Oo _X(1,1)\oplus \Oo _X(0,1)$ is the only bundle whose associated curve has bidegree $(1,1)$.
\item $\Oo _X(0,2)\oplus \Oo _X(1,0)$ is the only bundle whose associated curve has bidegree $(0,2)$.
\end{enumerate}
\end{lemma}

\begin{proof}
Since the rank is $2$, we get $\omega_C=\Oo_C(-1,-1)$ and $(e_2t_1t_2+e_1t_2^2)(-t_1-t_2)=(-e_1-e_2)t_1t_2^2$. It implies $p_a(C)=0$ and $e_1+e_2 = 2$. Since $s=1$, any bundle associated to a curve with bidegree either $(0,2)$ or $(1,1)$ has rank $2$.

Let $\Ee$ be any rank $2$ bundle associated to a smooth curve $C$ with bidegree $(1,1)$, i.e. $t_1t_2+t_2^2$ as its associated polynomial.  Since $C$ is connected and rational and $\deg (\Oo _C(0,1)) =1 =h^0(\Oo _X(0,1))-2$, there is a non-zero map $f: \Oo _X(1,1)\to \Ee$. Since $h^0(\Ii _C(0,0)) =h^0(\Ii _C(-1,1)) =0$, so $\mathrm{coker} (f)$ is torsion-free, i.e. $\mathrm{coker} (f) =\Ii _T(0,1)$ with either $T=\emptyset$ or $T$ a locally Cohen-Macaulay curve. Write $c_2(T) = 0$ if $T =\emptyset$, while call $c_2(T)$ the associated polynomial $at_1t_2+bt_2^2$ with $a:= \deg (\Oo _T(1,0))$ and $b:= \deg (\Oo _T(0,1))$ (if $T$ is not reduced, then for any line bundle $\Ll$ on $T$ the integer $\deg (\Ll )$ is the constant term of the Hilbert polynomial of $\Ll$ with respect to any very ample line bundle on $T$).
Since $c_2(\Oo _X(1,1)\oplus \Oo _X(0,1))=t_1t_2+t_2^2=c_2(\Ee)= c_2(T) +c_2(\Oo _X(1,1)\oplus \Oo _X(0,1))$, so we have $T =\emptyset$.

Now let $\Ee$ be any rank $2$ bundle associated to a smooth curve $C$ with bidegree $(0,2)$. Since $C$ is connected and rational and $\deg (\Oo _C(0,2)) = 4 = h^0(\Oo _X(0,2))-2$,
there is a non-zero map $f: \Oo _X(0,2)\to \Ee$. Since $h^0(\Ii _C(0,1)) =h^0(\Ii _C(-1,2)) =0$, so $\mathrm{coker} (f)$ is torsion-free, i.e. $\mathrm{coker} (f) =\Ii _T(0,1)$ with either $T=\emptyset$ or $T$ a curve. We have $T=\emptyset$, because $c_2(\Oo _X(0,2)\oplus \Oo _X(1,0))=2t_1t_2 =c_2(\Ee)$.
\end{proof}

\begin{lemma}\label{cccc1}
Let $\Ee$ be a globally generated vector bundle on $X$ with $c_1=(1,2)$ and no trivial factor whose associated curve is connected with bidegree $(1,2)$. Then we have $\Ee \cong \Oo_X(1,0)\oplus \Oo_X(0,1)^{\oplus 2}$.
\end{lemma}

\begin{proof}
Take any smooth and connected curves $C\subset X$ with bidegree $(1,2)$ and call $\Ee$ any bundle associated to $C$. Then $C$ is rational and so $\omega _C(1,1)$ is a line bundle of degree $1$. In particular $\Ee $ has rank $3$ and it is obtained using the Serre correspondence by the complete linear systems $|\omega _C(1,1)|$. The embedding $C\subset X$ is obtained by an isomorphism $\PP^1\to \PP^1$ of degree $1$, i.e. by the complete linear system $|\Oo _{\PP ^1}(1)|$, and by a morphism $j: \PP ^1 \to \PP^2$ induced by a base point-free linear subspace $V$ of $H^0(\Oo _{\PP^1}(2))$.

Assume for the moment that $C$ is not obtained by the complete linear system $|\Oo _{\PP ^1}(2)|$ and then we get that $j(\PP^1)$ is a line $L\subset \PP^2$. Then we have $C \subset \PP^1\times L$ as an element of $|\Oo _{\PP^1\times L}(2,1)|$ and so $\Ii _C(1,2)$ is not globally generated, a contradiction.

Now let $C'$ be another smooth and connected curve with bidegree $(1,2)$ and call $\Ee'$ any bundle associated to $C'$. We may assume that both $C$ and $C'$ are obtained by the complete linear system $|\Oo _{\PP ^1}(2)|$ and then we get that $C$ and $C'$ are projectively equivalent, i.e. there is $g\in \mathrm{Aut}(X)$ with $g(C') =C$. Since $\Ee $ and $\Ee '$ are obtained using the Serre correspondence by the complete linear systems $|\omega _C(1,1)|$ and $|\omega _{C'}(1,1)|$, we get $g^\ast (\Ee ') \cong \Ee$. The homogeneous vector bundle $\Oo_X(1,0)\oplus \Oo _X(0,1)^{\oplus 2}$ is obtained in this way and so we get $\Ee \cong \Oo_X(1,0)\oplus \Oo _X(0,1)^{\oplus 2}$.
\end{proof}

\begin{lemma}\label{cca2}
Any pair in the set $\{ (x,4), (y,3)~|~0 \le x\le 3~,~0\le y \le 1\}$ cannot be a bidegree of a curve associated to a globally generated bundle.
\end{lemma}

\begin{proof}
\quad{(a)} Assume that a bidegree $(x,4)$ with $0\le x\le 3$, is achieved by a bundle $\Ee$ with $C$ as an associated curve. By Remark \ref{cca0} $C$ is a smooth and connected curve of genus $3$. The case $x=0$ does not occur, because there are $P\in \PP^1$ and $C'\in |\Oo _{\PP^2}(4)|$ such that $C = \{P\}\times C'$ and so even $\Ii _C(3,3)$ is not globally generated. Thus we have $h^1(\Oo _C(1,1))=0$. The case $x=1$ does not occur, because no curve of genus $3$ has a very ample line bundle of degree $5$ (or because ${\pi _1}_{\vert_C}: C\to \PP^1$ is not an isomorphism). Now assume $2\le x\le 3$. Since $h^1(\Oo _C(1,1)) =0$, we have $h^0(\Oo _C(1,1)) <6$. Therefore $h^0(\Ii _C(1,1)) >0$ and so $h^0(\Ee (0,-1)) >0$.
Let $f: \Oo _X(0,1)\to \Ee$ be a general map. Since $h^0(\Ii _C(0,1)) =h^0(\Ii _C(1,0)) =0$, so $\mathrm{coker} (f)$ is torsion-free and so $\mathrm{coker} (f) \cong \Ii _T(1,1)$ with either
$T =\emptyset$ or $T$ a curve. We have $T\ne \emptyset$, because $\Ee \ne \Oo _X(0,1)\oplus \Oo _X(1,1)$. Setting $a:= \deg (\Oo _T(1,0))$ and $b:= \deg (\Oo _T(0,1))$, we have $4t_1t_2 + xt_2^2 = c_2(\Ee ) = c_2(\Oo _X(0,1)\oplus \Oo _X(1,1))+at_1t_2+bt_2^2$ and so $(a,b)=(3,x-1)$. Since $\Ii_T(1,1))$ is globally generated and $\deg (X) =3$, we get  $\deg (T) \le 3$, i.e. $a+b =x+2 \le 3$, a contradiction.

\quad{(b)} Let $C$ be a smooth curve of bidegree $(y,3)$ with $0\le y\le 1$. Assume that $\Ii _C(2,1)$ is globally generated and $\Ee$ is an associated bundle with $c_2(\Ee) =3t_1t_2+yt_2^2$. If
$y= 0$, then there are $P\in \PP^1$ and $C'\in |\Oo _{\PP^2}(3)|$ such that $C = \{P\}\times C'$ and so even $\Ii _C(2,2)$ is not globally generated. The case $y=1$ is not possible, because
$C$ has genus $1$ and so ${\pi _1}_{\vert_C}$ is not an isomorphism.
\end{proof}

\begin{lemma}\label{cca2.0}
The bidegree $(4,3)$ is not the bidegree of a smooth curve $C$ with $\Ii _C(1,2)$ globally generated.
\end{lemma}

\begin{proof}
Assume that $C$ is curve of bidegree $(4,3)$ with $\Ii _C(1,2)$ globally generated. Remark \ref{cca0} gives that $C$ is connected with genus $1$. Take $Y =C\cup D$. Since $D$ is a curve with bidegree $(0,1)$, it is a line of type $(0,1)$. Since $p_a(Y)=3$, we have $\deg (D\cap C)>1$ and so $D$ is contained in the base locus
of $\Ii _C(1,2)$, a contradiction.
\end{proof}

\begin{lemma}\label{cca2.1}
For each $0\le x \le 4$, the bidegree $(x,2)$ is realized by a globally generated bundle.
\end{lemma}

\begin{proof}
Fix a general set $S\subset \PP^2$ with $\sharp (S) =x$ and set $T:= \PP^1\times S$. $T$ is a union of $x$ disjoint lines of type $(1,0)$. The line bundle $\omega _T(3,1)$ is globally generated and so a general section of it defines an exact sequence
\begin{equation}\label{eqcca3}
0 \to \Oo _X(1,0) \to \Ff \to \Ii _T(0,2)\to 0
\end{equation}
with $\Ff$ a reflexive sheaf of rank $2$ with $c_3(\Ff )=\deg (\omega _T(3,1)) = x$. Since $x\le 4$ and no three points in $S$ are collinear, $\Ii _{S,\PP^2}(2)$ is globally generated. Thus the sequence (\ref{eqcca3}) gives that $\Ff$ is globally generated. Let
\begin{equation}\label{eqcca4}
0 \to \Oo _X \to \Ff \to \Ii _C(1,2)\to 0
\end{equation}
be the exact sequence associated to a general section of $\Ff$. Since $T$ is smooth, so is $C$ (see \cite{Hi}, \cite[Theorem 3.2]{hh}, \cite{bo}). Since $\Ff$ is globally generated, so is $\Ii _C(1,2)$. Since $xt_2^2$ is the polynomial associated to $T$, then (\ref{eqcca3}) gives $c_2(\Ff ) = 2t_1t_2+xt_2^2$. Therefore $C$ has bidegree $(x,2)$.
\end{proof}

\begin{lemma}\label{cca3+}
There exist subschemes $T \subset X$ such that $\omega_T(1,3)$ and $\Ii_T(1,1)$ are globally generated as in the following list:
\begin{enumerate}
\item a line of bidegree $(0,1)$; $(m,n)=(1,4)$;
\item a smooth conic of bidegree $(1,1)$; $(m,n)=(2,3)$;
\item a smooth rational curve of bidegree $(1,2)$; $(m,n)=(6,2)$,
\end{enumerate}
where $(m,n):=(\deg (\omega_T(1,3)), h^0(\Ii_T(1,1)))$.
\end{lemma}

\begin{proof}
For each line $T$, $\Ii _T(1,1)$ is globally generated. If $T$ has bidegree $(0,1)$, then $\deg (\omega _T(1,3)) =1$. Since $T$ is rational, so $\omega _T(1,3)$ is globally generated. Obviously we have $h^0(\Ii _T(1,1)) =4$.

Since $(t_1+t_2)(t_1+t_2) = 2t_1t_2+t_2^2$, so a general complete intersection $T$ of two elements of $|\Oo _X(1,1)|$ is smooth rational curve with bidegree $(2,1)$. As $T$ is a complete intersection curve, we get that $\Ii _T(1,1)$ is globally generated with $h^0(\Ii _T(1,1)) =2$. Since $T$ is smooth and rational, we have $\deg (\omega _T(1,3)) = 1+6-2=5$ and $\omega _T(1,3)$ is globally generated. Fix a line $L\subset \PP^2$. The quadric $\PP^1\times L$ is a hyperplane section in $|\Oo _X(1,1)|$. Take as a smooth conic of type $(1,1)$ a smooth element of $|\Oo _{\PP^1\times L}(1,1)|$.
\end{proof}

\begin{lemma}\label{cca3+++}
The bidegrees $(1,2)$, $(2,2)$ and $(2,3)$ are associated to some globally generated bundle $\Ee$.
\end{lemma}

\begin{proof}
Take $T$ as in Lemma \ref{cca3+}. A general section of $\omega _T(1,3)$ gives a reflexive sheaf of rank $2$ on $X$ fitting in the exact sequence
\begin{equation}\label{eqcc1}
0\to \Oo _X(0,1)\to \Ff \to \Ii _T(1,1)\to 0
\end{equation}
with $c_3(\Ff ) = \deg (\omega _T(1,3)) = 1$, $h^0(\Ff ) = 2+h^0(\Ii _T(1,1))$ and $\Ff$ globally generated. A general section of $\Ff$ induces and exact sequence
\begin{equation}\label{eqcc2}
0 \to \Oo _X \to \Ff \to \Ii _C(1,2)\to 0
\end{equation}
with $C$ a locally Cohen-Macaulay curve. Since $\Ff$ is globally generated, so is $\Ii _C(1,2)$. Since $T$ is smooth, $C$ is also smooth (see \cite{Hi}, \cite[Theorem 3.2]{hh}, \cite{bo}). Therefore $C$ is one of the curve in Remark \ref{cca0} and we only need to check which bidegree is realized by each $T$. The integer $c_3(\Ff ) =\deg (\omega _T(1,3))$
is the integer $\deg (\omega _C(1,1))$, because $\Ff$ is also obtained from a section of $\omega _C(1,1)$. Thus we have
$$\deg (\omega_C(1,1))= \left\{
                                           \begin{array}{lll}
                                             1, & \hbox{if $T$ is a line;}\\
                                             2, & \hbox{if $T$ is a smooth conic;} \\
                                             6, & \hbox{if $T$ is a smooth rational curve of bidegree $(1,2)$.}
                                           \end{array}
                                         \right.$$
If $C$ has genus $3$, then Remark \ref{cca0} gives that it has bidegree $(4,4)$ and so $\deg (\omega _C(1,1)) = 10$. If $C$ has genus $1$, then $\deg (\omega _C(1,1))
= e_1+3$. If $C$ is rational, then $\deg (\omega _C(1,1)) = e_1+e_2-2$. Since $e_2 \le 2$ for a rational case, we see that the curve $T$ with bidegree $(1,2)$ gives a rational curve $C$ with bidegree $(2,3)$, the line $T$ gives a rational curve $C$ with $e_1+e_2=3$ and the smooth conic $T$ gives a rational curve with
$e_1+e_2=4$.

Assume that $T$ is a line of bidegree $(0,1)$ (i.e. it is associated to $t_1t_2$). By (\ref{eqcc1}) we get that $c_2(\Ff )$ is represented by $t_1t_2+ t_2(t_1+t_2)= 2t_1t_2+t_2^2$ and so (\ref{eqcc2}) gives that $C$ has bidegree $(1,2)$. Now assume that $T$ is a conic of bidegree $(1,1)$, i.e. it is associated to $t_1t_2 +t_2^2$. We get
$c_2(\Ff ) = t_1t_2+t_2^2 + t_2(t_1+t_2) = 2t_1t_2+2t_2^2$, i.e. $C$ has bidegree $(2,2)$. Now assume that $T$ has bidegree $(1,2)$, i.e. it is associated
to $2t_1t_2+t_2^2$. Since $c_2(\Ff) =3t_1t_2+2t_2^2$, we get that $C$ has bidegree $(2,3)$.
\end{proof}

\begin{lemma}\label{cca3}
The bidegree $(3,3)$ is realized by a globally generated bundle.
\end{lemma}

\begin{proof}
Every smooth and connected  curve $C\subset X$ with bidegree $(3,3)$ is obtained in the following way: Fix a smooth elliptic curve $C$ and two line bundles $\Ll _1$ and $\Ll _2$ of degree $3$ on $C$. Since $3 = 2p_a( C)+1$, these line bundles are very ample and non-special with $h^0(\Ll _i)=3$. Take $V_2:= H^0(\Ll _2)$, while take as $V_1$ any $2$-dimensional linear subspace of $H^0(\Ll _1)$ without base points. In particular ${\pi _2}_{\vert_C}$ is always an embedding and so $\deg (J\cap C) \le 1$ for each line $J\subset X$ of bidegree $(1,0)$. Fix any smooth and connected $C\subset X$ with bidegree $(3,3)$. Since $h^0(\Oo _C(3,3)) =6$ by the Rieman-Roch theorem, we have $h^0(\Ii _C(1,2)) \ge 3$. We easily see that $h^0(\Ii _C(1,1)) =0$, but we only need it for at least one curve $C$.

\quad {\emph {Claim 1:}} There are only finitely many lines $J$ of type $(0,1)$ with $\deg (J\cap C)\ge 2$.

\quad {\emph {Proof of Claim 1:}} If Claim 1 fails, then $C$ is contained in an integral surface $S\subset X$ ruled by lines of type $(0,1)$ such that $\deg (J\cap C) \ge 2$ for a general line $J$ of the ruling of $S$. The map ${\pi _1}_{\vert_S}$ shows that the ruling of $S$ is induced by $\pi_1$.
Since $C$ has bidegree $(3,3)$, we get $\deg (J\cap C)=3$ for a general line $J$ of the ruling of $S$ and so we have $h^0(\Ii _C(1,2)) = h^0(\Ii _S(1,2))$. Since $\deg (\Oo _C(0,1)) >2$, we have $h^0(\Ii _C(0,2)) =0$. Since $h^0(\Ii _C(1,1)) =0$, we have $S\in |\Oo _X(1,2)|$. Hence $h^0(\Ii _S(1,2)) =1$, a contradiction. \qed

Fix a general line $L\subset \PP^2$ and then, by Claim 1, the set $A:= C\cap (\PP^1\times L)$ is formed by three points such that no two of them are contained in a line of the quadric $\PP^1\times L$. Thus the unique $D\in |\Ii _{A,\PP^1\times L}(1,1)|$ is a smooth rational curve and it has bidegree $(1,1)$ as a curve of $X$. Since $A =C\cap D$ as schemes, the curve $C\cup D$ is connected and nodal with $p_a(C\cup D) =3$ and bidegree $(4,4)$. Since $h^0(\Oo _D(1,2)) =4$ and $\sharp (A)=3$, we have $h^0(\Ii _{C\cup D}(1,2)) \ge h^0(\Ii _C(1,2)) -1\ge 2$. Since $C\cup D$ has bidegree $(4,4)$ and $C$ is contained in no element of $|\Oo _X(1,1)|$ or $|\Oo _X(0,2)|$, so $C\cup D$ is the complete intersection of two elements of $|\Oo _X(1,2)|$ and $\Ii _{C\cup D}(1,2)$ is globally generated. In particular $\Ii _C(1,2)$ is globally generated outside the points of $D$. Fix general lines $L_i'\subset \PP^2$ for $i=1,2$, and set $A_i:= C\cap L_i$ and $D_i\in |\Ii _{A_i}(1,1)|$. Since  $\Ii _{C\cup D_i}(1,2)$ is globally generated, so $\Ii _C(1,2)$ is globally generated outside the points of $D_i$. Since we may assume $L\cap L_1\cap L_2 =\emptyset$, so we have $D\cap D_1\cap D_2=\emptyset$ and hence $\Ii _C(1,2)$ is globally generated.
\end{proof}

\begin{lemma}\label{cca2.2}
A bidegree in the set $\{(x,1)~|~ 0 \le x \le 4\}$ is realized by a globally generated bundle if and only if $x\in \{1,2\}$.
\end{lemma}

\begin{proof}
Assume that a curve $C$ with bidegree $(x,1)$ is realized by a globally generated bundle $\Ee$, say of rank $3$. By Remark \ref{cca0} $C$ is smooth and rational. Taking the quotient of $\Ee$ by the image of a general section, we obtain a globally generated reflexive sheaf $\Ff$ of rank $2$. Since $h^0(\Oo _C(0,1)) =2<h^0(\Oo _X(0,1))$
and $h^1(\Oo _X(-1,-1)) =0$, we get $h^0(\Ff (-1,-1)) >0$. Since $h^0(\Ii _C(-1,1)) =h^0(\Ii _C(0,0)) =0$, a general map $f: \Oo _X(1,1) \to \Ff$ has torsion-free cokernel and so $\mathrm{coker} (f) \cong \Ii _T(0,1)$ with either $T=\emptyset$ or $T$ a locally Cohen-Macaulay curve, i.e. $\Ff$ fits in an exact sequence
\begin{equation}\label{eqcca5}
0 \to \Oo _X(1,1) \to \Ff \to \Ii _T(0,1)\to 0.
\end{equation}
Since $\Ff$ is globally generated, so is $\Ii_T(0,1)$ by (\ref{eqcca5}). If $L_0, L_1, L_2\subset \PP^2$ are three lines, then we have
$$(\PP^1 \times L_0) \cap (\PP^1 \times L_i)=\PP^1 \times (L_0 \cap L_i)$$
for $i=1,2$. Thus we have either $T=\emptyset$ or $T$ is a line of type $(1,0)$.

If $T =\emptyset$, then $c_2(\Ff ) =t_1t_2+t_2^2$ and so $C$ has bidegree $(1,1)$. This case is realized by $\Oo _X(1,1)\oplus \Oo _X(0,1)$. Let $T$ be a line of bidegree $(1,0)$. Since $\omega _T(3,3)$ is globally generated, a general section of $\omega _T(3,3)$ induces an extension (\ref{eqcca5}). Since $T$ is a complete intersection of two elements of $|\Oo _X(0,1)|$, $\Ff$ is globally generated. Thus a general section of $\Ff$ induces an exact sequence (\ref{eqcca3}) with $\Ii _C(1,2)$ globally generated and $c_2(\Ff)=(t_1+t_2)t_2+t_2^2$, i.e. $C$ has bidegree $(2,1)$. We may take as $C$ a smooth curve (\cite{Hi}, \cite[Theorem 3.2]{hh}, \cite{bo}).
\end{proof}

\begin{proposition}\label{ccca2}
Let $\Ee$ be a globally generated vector bundle of rank $r\ge 2$ on $X$ with $c_1=(1,2)$ and no trivial factor, associated to a curve of bidegree $(2,2)$. Then we have $r\in \{3,4\}$ and $c_3(\Ee )=2$. Indeed we have
\begin{enumerate}
\item $\Ee \cong \pi _2^*(T\PP^2(-1)) \oplus \Oo _X(1,0)\oplus \Oo _X(0,1)$ if $r=4$;
\item $\Ee$ is the cokernel of a non-zero map $$\Oo _X\to  \pi _2^*(T\PP^2(-1)) \oplus \Oo _X(1,0)\oplus \Oo _X(0,1)$$ with locally free cokernel if $r=3$.
\end{enumerate}
\end{proposition}

\begin{proof}
The bundle $\pi _2^*(T\PP^2(-1))\oplus \Oo _X(1,0)\oplus \Oo _X(0,1)$ has $c_3 =2$ and it is associated to a curve of bidegree $(2,2)$, because
$(1+t_1)(1+t_2)(1+t_2+t_2^2) = 1+t_1+2t_2+2t_1t_2+2t_2^2 +2t_1t_2^3$. Let $\HH$ be the open subset of the Hilbert scheme of $X$ parametrizing all smooth and connected curves of bidegree $(2,2)$ associated to a pair $(u_1,u_2)$, where $u_1: \PP^1 \to \PP^1$ is a degree $2$ morphism and $u_2: \PP^1 \to \PP^2$ is an embedding of $\PP^1$ as a smooth conic. By Lemma \ref{ca3.1} these are the only curves with bidegree $(2,2)$ associated to a globally generated bundle. Since $u_2$ is given by the complete linear system $|\Oo _{\PP^2}(2)|$, so $\HH$ is irreducible.

Fix any $C\in \HH$ and let $N_C$ be the normal bundle of the inclusion $C\subset X$. Since $X$ is homogeneous, $N_C$ is spanned. Since $C\cong \PP^1$, we get $h^1(N_C) =0$ and so $\HH$ is smooth at $C$ of dimension
$$h^0(N_C) = \deg ({\omega _X^\vee}_{\vert_C}) -\deg (T\PP^1)+\mathrm{rank}(N_C)= 10.$$
Since $C$ is an arbitrary element of $\HH$ and $\HH$ is irreducible, to see that $\HH$ is an orbit by the action of $\mathrm{Aut}(X)$ on $\HH$, it is sufficient to prove that the orbit of $C$ has dimension $10$. Since $\mathrm{Aut}(X) = \mathrm{Aut}(\PP^1)\times \mathrm{Aut}(\PP^2)$ has dimension
$11$, it is sufficient to prove that the stabilizer of $C$ has at most dimension $1$, i.e. for a general $O\in C$ there are only finitely elements $g\in \mathrm{Aut}(X)$ with $g({C}) =C$ and $g(O)=O$. Take $(u_1,u_2)$ associated to $C$ and call $P,P'$ the two ramification points of the degree $2$ morphism $u_1: \PP^1\to \PP^1$. Take as $O$ any point of $\PP^1\setminus \{P,P'\}$. Since the identity is the only element of $\mathrm{Aut}(\PP^1)$ fixing $3$ points of $\PP^1$, there are only finitely many $h\in \mathrm{Aut}(\PP^1)$ with $h(O) =O$ and $h(\{P,P'\}) = \{P,P'\}$.

Now assume $r=4$. Since $\pi _2^*(T\PP^2(-1)) \oplus \Oo _X(1,0)\oplus \Oo _X(0,1)$ is invariant under the pull-back $f^*$ for any $f\in \mathrm{Aut}(X)$ and $\HH$ is an orbit of $\mathrm{Aut}(X)$, we get $\Ee \cong \pi _2^*(T\PP^2(-1)) \oplus \Oo _X(1,0)\oplus \Oo _X(0,1)$.

Now assume $r=3$. Take any $C\in \HH$ associated to $\Ee$. Since $h^0(\omega _C(1,1)) =3$, the $\Ext^{\bullet}$-sequence associated to (\ref{equ+}) gives $h^1(\Ee ^\vee )=1$ and so $\Ee$ fits into a non-split exact sequence
$$0 \to \Oo _X\to \Gg \to \Ee \to 0.$$
Since we proved that $\Gg \cong \pi _2^*(T\PP^2(-1)) \oplus \Oo _X(1,0)\oplus \Oo _X(0,1)$, we get part (2).
\end{proof}

\begin{theorem}\label{cca5}
If $\Ee$ is a globally generated vector bundle of rank $r\ge 2$ on $X$ with $c_1=(1,2)$ and no trivial factor, then we have
\begin{enumerate}
\item if $r=2$, then $\Ee$ is isomorphic to either
$$\Oo _X(0,1)\oplus \Oo _X(1,1) ~~\text{   or      }~~~\Oo _X(1,0)\oplus \Oo_X(0,2).$$
\item if $r\ge 3$, then the quadruples $(r,s;a_1,a_2)$ corresponding to $\Ee$ that realizes a curve $C$ with $s$ connected components and bidegree $(a_1,a_2)$ are the following ones:
\begin{itemize}
\item [(i)] $(s;a_1, a_2)=(1;4,4)$; $C$ the complete intersection curve with genus $3$; $3\le r \le 11$;
\item[(ii)] $(s;a_1, a_2)=(1,2,3)$; $3\le r \le 8$;
\item [(iii)] $(s;a_1, a_2)=(1;3,3)$; $3\le r \le 7$;
\item [(iv)] $(s;a_1, a_2)=(1;x,2)$ for $1 \le x\le 4$; $3 \le r \le x+2$;\\$x=1\Leftrightarrow\Ee \cong \Oo _X(1,0)\oplus \Oo _X(0,1)^{\oplus 2}$;
\item [(v)] $(s; a_1, a_2;r)=(1;2,1;3)$.
\end{itemize}
\end{enumerate}
\end{theorem}

\begin{proof}
In Remark \ref{cca0} we look at the curves $C$ with $\Ii _C(1,2)$ globally generated and with $\omega _C(1,1)$ globally generated. In all cases we have $s=1$, i.e. $C$ is connected. The ones giving bundles of rank $2$ are the connected curves with bidegree $(1,1)$ and $(0,2)$. Lemma \ref{cca1} says that they only give the bundles in the assertion (1). Since $s=1$, these bidegrees do not give higher rank bundles with no trivial factor. Now assume $r>2$. Each smooth curve $C$ with $\Ii _C(1,2)$ spanned, $\omega _C(1,1)$ globally generated and $\omega _C(1,1) \ne \Oo _C$ gives a spanned bundle of rank $r$ with no trivial factor if and only if $3 \le r \le h^0(\omega _C(1,1))$. Among the bidegrees listed in Remark \ref{cca0} some of them give bundles (see Lemmas  \ref{cca2.1}, \ref{cca3+++}, \ref{cca3}, \ref{cca2.2}, but the bidegree $(0,2)$ corresponds to $\Oo _X(1,0)\oplus \Oo_X(0,2)$ and the bidegree $(1,2)$ to $ \Oo _X(1,0)\oplus \Oo _X(0,1)^{\oplus 2}$ by Lemma \ref{cccc1})
and all the remaining ones are excluded (see Lemmas \ref{cca2}, \ref{cca2.0}, \ref{cca2.2}).
\end{proof}

\begin{remark}
See Proposition \ref{ccca2} for a description of the case $(s;a_1,a_2) =(2;2,2)$.
\end{remark}


\bibliographystyle{amsplain}

\begin{thebibliography}{10}

\bibitem{am}
C.~Anghel and N.~Manolache, \emph{Globally generated vector bundles on {$\Bbb
  P^n$} with {$c_1=3$}}, Math. Nachr. \textbf{286} (2013), no.~14-15,
  1407--1423. \MR{3119690}

\bibitem{Arrondo}
E.~Arrondo, \emph{A home-made {H}artshorne-{S}erre correspondence}, Rev. Mat.
  Complut. \textbf{20} (2007), no.~2, 423--443. \MR{2351117 (2008g:14084)}


\bibitem{BHM}
E. Ballico, S. Huh,  and F.~Malaspina, \emph{Globally generated vector bundles of rank $2$ on a smooth
  quadric threefold}, J. Pure Appl. Algebra \textbf{218} (2014), no.~2,
  197--207. \MR{3120621}

\bibitem{BHM+++}
\bysame, \emph{Globally generated vector bundles on complete intersection
  calabi-yau threefolds}, preprint, arXiv:1411.6183v1 [math.AG] (2014).

\bibitem{bo}
G.~Bolondi, \emph{Smoothing curves by reflexive sheaves}, Proc. Amer. Math.
  Soc. \textbf{102} (1988), no.~4, 797--803. \MR{934845 (89e:14033)}

\bibitem{CFM0}
G.~Casnati, D.~Faenzi, and F.~Malaspina, \emph{Rank two acm bundles on the del
  pezzo threefold with picard number $3$}, preprint, arXiv:1306.6008v1
  [math.AG] (2013).

\bibitem{CFM}
\bysame, \emph{Moduli spaces of rank two acm bundles on the segre product of
  three projective lines}, preprint, arXiv:1404.1188v2 [math.AG] (2014).

\bibitem{CMP}
L.~Costa, R.~Mir\'o-Roig and J.~Pons-Llopis, \emph{The representation type of Segre varieties}, Adv. Math. \textbf{230} (2012), no.~4--6, 1995--2013.

\bibitem{hh}
R.~Hartshorne and A.~Hirschowitz, \emph{Nouvelles courbes de bon genre dans
  l'espace projectif}, Math. Ann. \textbf{280} (1988), no.~3, 353--367.
  \MR{936316 (89d:14043)}

\bibitem{Hi}
A.~Hirschowitz, \emph{Existence de faisceaux r\'eflexifs de rang deux sur
  {${\bf P}^3$} \`a bonne cohomologie}, Inst. Hautes \'Etudes Sci. Publ. Math.
  (1988), no.~66, 105--137. \MR{932136 (89c:14019)}

\bibitem{man}
N.~Manolache, \emph{Globally generated vector bundles on $\mathbb{P}^3$ with
  $c_1=3$}, Preprint, arXiv:1202.5988 [math.AG], 2012.

\bibitem{sierra}
J.~C. Sierra, \emph{A degree bound for globally generated vector bundles},
  Math. Z. \textbf{262} (2009), no.~3, 517--525. \MR{2506304 (2010i:14076)}

\bibitem{SU}
J.~C. Sierra and L.~Ugaglia, \emph{On globally generated vector bundles on
  projective spaces}, J. Pure Appl. Algebra \textbf{213} (2009), no.~11,
  2141--2146. \MR{2533312 (2010d:14062)}


\end{thebibliography}

\providecommand{\bysame}{\leavevmode\hbox to3em{\hrulefill}\thinspace}
\providecommand{\MR}{\relax\ifhmode\unskip\space\fi MR }
\providecommand{\MRhref}[2]{%
  \href{http://www.ams.org/mathscinet-getitem?mr=#1}{#2}
}
\providecommand{\href}[2]{#2}

\end{document}